\documentclass[10pt,english,online]{smfart}
\usepackage[T1]{fontenc}
\usepackage[francais]{babel}
\usepackage{tikz-cd}
\tikzcdset{arrow style=math font}
\usepackage{amssymb,url,xspace,smfthm,amsmath,amsfonts,enumitem}

\title[Logarithmic good reduction and the index]{Logarithmic good reduction and the index}
\author{Kentaro Mitsui}
\address{Department of Mathematics, Graduate School of Science, Kobe University, Hyogo 657--8501, Japan}
\email{mitsui@math.kobe-u.ac.jp}
\author{Arne Smeets}
\address{Radboud Universiteit Nijmegen, IMAPP, Heyendaalseweg 135, 6525 AJ Nijmegen, The Netherlands \emph{and} University of Leuven, Departement Wiskunde, Celestijnenlaan 200B, 3001 Heverlee, Belgium} 
\thanks{The first named author gratefully acknowledges financial support of JSPS and CNRS. The second named author thanks FWO Vlaanderen for financial support and the Max Planck Institute for Mathematics for its hospitality. We also acknowledge support by the European Research Council. The details are given at the end.}
\email{arnesmeets@gmail.com}

\newcommand{\pp}{p} 
\newcommand{\tp}{{p'}} 

\begin{document}
\maketitle

\begin{abstract} Let $K$ be the fraction field of a complete discrete valuation ring, with algebraically closed residue field of characteristic $p > 0$. This paper studies the index of a smooth, proper $K$-variety $X$ with logarithmic good reduction. We prove that it is prime to $p$ in `most' cases, for example if the Euler number of $X$ does not vanish, but (perhaps surprisingly) not always. We also fully characterise curves of genus $1$ with logarithmic good reduction, thereby completing classical results of T. Saito and Stix  valid for curves of genus at least $2$.
\end{abstract}

\section{Introduction}

\subsection{Context} Let $K$ be the fraction field of a discrete valuation ring, with perfect residue field $k$ of characteristic $p \geq 0$. Fix a separable closure $K^s$ of $K$, and denote by $K^t$ the maximal tamely ramified extension of $K$ contained in $K^s$. Let $X$ be a smooth, proper $K$-variety (all $K$-varieties are assumed to be geometrically integral over $K$).

The \emph{index} $\iota(X)$ is the smallest positive degree of a zero-cycle on $X$, or equivalently, the greatest common divisor of the degrees of all finite field extensions $L/K$ for which $X(L) \neq \emptyset$. This is an important arithmetic invariant of $X$, which has been well studied; let us mention three interesting and recent contributions to the topic, in chronological order. Denote by $S$ the spectrum of the valuation ring of $K$.

Gabber--Liu--Lorenzini show in \cite{GLL} how, over arbitrary Henselian discretely valued fields, $\iota(X)$ can be computed from the special fibre of a proper, regular model of $X$ over $S$. Their approach is based on intersection theory and moving lemmata. 

Esnault--Levine--Wittenberg study the index in \cite{ELW} using Euler characteristics of coherent sheaves on $X$. Among other results, they prove the statement that if $X$ is a $K$-variety which is rationally connected and if moreover $p = 0$ or $p > \dim X + 1$, then $\iota(X) = 1$.

In \cite{KN}, Kesteloot--Nicaise introduce the \emph{specialisation index} $\iota_\mathrm{sp}(X)$, an invariant which refines the index. They explain how to compute it starting from a log regular model of $X$ over $S$. Their work suggests the following natural question (also mentioned in \cite[Introduction]{Sm}).

\subsection{Question} \label{sec:question} Assume that $p > 0$. It is well-known that every positive integer prime to $p$ appears as the index of some smooth, proper $K$-variety $X$ with \emph{logarithmic good reduction}, i.e.\ with a proper, log smooth model of $X$ over $S$. Conversely, if $X$ is a smooth, proper $K$-variety with logarithmic good reduction, is $\iota(X)$ necessarily prime to $p$?
This is equivalent to asking whether $X(K^t) \neq \emptyset$ if $X$ has logarithmic good reduction. 

It may seem natural to expect an affirmative answer to this question, by analogy with Hensel's lemma for smooth families in the classical sense. This note will 
show that this expectation is too na\"ive: the answer is affirmative in ``most'' cases, but not always. 

We should mention that the existence of quasi-sections in logarithmic geometry has been studied in a more general setting by Nakayama in \cite{Na2}; in this note, we sacrifice generality in order to obtain precise answers to the question in a concrete, geometric setting.

\subsection{Results} We obtain two main results: a positive statement which is rather general, and some specific examples showing that this positive statement is essentially optimal.
 
Our first main result (proven in \S \ref{section3}) gives sufficient conditions for the existence of a $K^t$-point. Denote by $\chi(X)$ the $\ell$-adic Euler characteristic $\sum_{i \geq 0} (-1)^i \dim_{\mathbf{Q}_\ell} H^i(X \times_K K^s, \mathbf{Q}_\ell)$, where $\ell \neq p$ is prime; this is an integer independent of the choice of $\ell$.

\begin{theo} \label{maintheorem} Let $X$ be a smooth, proper $K$-variety with logarithmic good reduction. Assume that $\chi(X) \neq 0$. Then $X(K^t) \neq \emptyset$, i.e.\ $\iota(X)$ is prime to $p$. \end{theo}

The topological assumption on the ($\ell$-adic) Euler number may seem a bit strange, but it turns out to be necessary since Theorem \ref{maintheorem} breaks down when $\chi(X) = 0$, even in the case of curves. Indeed, our second main result (proven in \S \ref{section5}) yields a full characterisation of curves of genus $1$ with logarithmic good reduction, completing the results obtained previously by T. Saito and Stix for curves of genus at least $2$ in \cite{Sa, Sa2, St}. 

To state our result, recall that the \emph{period} of a curve $C$ of genus $1$ over $K$ is defined as the order of the class of $C$ in $H^1(K,J)$, where $J$ is the Jacobian of $C$. Since we assume the residue field $k$ to be algebraically closed, the Brauer group of $K$ vanishes, and hence, the period is actually equal to the index $\iota(C)$ by a result of Lichtenbaum \cite[Theorem 1]{Lichtenbaum}.
\begin{theo} \label{genus1theorem} Let $C$ be a curve of genus $1$ over $K$. Denote by $m$ its period and by $\mathcal{C}$ its minimal regular model over $S$. The curve $C$ has logarithmic good reduction if and only if both of the following conditions are satisfied: \begin{itemize} \item[(a)] the Galois action on $H^1(C \times_K K^s, \mathbf{Q}_\ell)$ is tamely ramified; \item[(b)] if $p \mid m$, then the Jacobian of $C$ has good reduction, and $\mathcal{C}$ is cohomologically flat over $S$. \end{itemize}
\end{theo}

Recall that the (by now classical) results of T. Saito and Stix say that a curve $C$ of genus at least $2$ over $K$ has logarithmic good reduction if and only if the Galois action on $H^1(C \times_K K^s, \mathbf{Q}_\ell)$ is tamely ramified. The same result is known to hold for elliptic curves, but condition (b) in our theorem shows that the situation is more delicate for curves of genus $1$ which do not have a rational point. The whole crux of Theorem \ref{genus1theorem} is of course that there really exist curves $C$ of genus $1$ over $K$ satisfying both conditions (a) and (b) and for which $C(K^t) = \emptyset$ (or equivalently, $p \mid m$): this follows from \cite[9.4.1.(iii)]{Ra}. For the construction of such examples, the interested reader can consult the work of Katsura--Ueno \cite{KU} and Harbourne--Lang \cite{HL} on so-called \emph{tame} and \emph{wild} fibres of elliptic surfaces in positive characteristic; the construction in \cite[Example 4.9]{KU} is particularly elegant.

\subsection{Structure of the paper} In \S \ref{section2}, we will give a simple geometric description, using differential forms, for the log smooth locus of a log regular scheme over a discrete valuation ring. This description will subsequently be used in \S \ref{section3} to prove that the geometry of a log smooth degeneration without any tamely ramified quasi-sections must be strongly restricted. Based on these considerations, we then prove Theorem \ref{maintheorem}. 

In \S \ref{section4}, we define an invariant which, still in the  spirit of the criterion from \S \ref{section2}, allows to measure the ``defect'' of log smoothness in the absence of tamely ramified quasi-sections. This proves useful in \S \ref{section5} for the trickier cases of the study of curves of genus $1$, that is, for the study of torsors under elliptic curves with good reduction: it allows us to separate those torsors with logarithmic good reduction from those without.

\subsection{Notation and conventions} For generalities on logarithmic geometry, we refer to the foundational papers by Kazuya Kato, or to \cite[\S 2]{Sm} for a short summary of the notions needed in this paper.

Given the spectrum $S$ of a discrete valuation ring and a flat $S$-scheme $\mathcal{X}$, we denote by $\mathcal{X}^\dag$ the log scheme obtained by equipping $\mathcal{X}$ with the natural log structure induced by the special fibre $\mathcal{X}_s$. If $U$ denotes the generic fibre of $\mathcal{X} \to S$, and if $j: U \to \mathcal{X}$ denotes the corresponding open immersion, then this log structure is given by the inclusion $\mathcal{M}_X := \mathcal{O}_X \cap j_*\mathcal{O}_U^\times \hookrightarrow \mathcal{O}_X$. A scheme $X$ smooth and proper over the generic point of $S$ is said to have \emph{logarithmic good reduction} if there exists a proper model $\mathcal{X}$ of $X$ over $S$ with the property that $\mathcal{X}^\dag$ is log smooth over $S^\dag$.

Given a monoid $P$, we denote by $P^\mathrm{gp}$ its group envelope, by $P^\times$ the subgroup of invertible elements, and by $P^\sharp$ the associated sharp monoid $P/P^\times$.

\section{The log smooth locus of a log regular model} \label{section2}

Let $R$ be a discrete valuation ring, with fraction field $K$ and perfect residue field $k$ of characteristic $p \geq 0$.  Let $S = \mathrm{Spec}\,R$, and let $\mathcal{X}$ be a flat $S$-scheme of finite type, with smooth geometrically integral generic fibre and special fibre $\mathcal{X}_s$. Assume that $\mathcal{X}^\dag$ is log regular. The goal of this section is to understand the log smooth locus of the induced morphism $\mathcal{X}^\dag \to S^\dag$ of log schemes.

Denote by $F(\mathcal{X}^\dag)$ the Kato fan associated with $\mathcal{X}^\dag$. There is a continuous map of monoidal spaces $\Pi\colon\mathcal{X}^\dag \to F(\mathcal{X}^\dag)$, which determines a standard stratification of $\mathcal{X}$ into finitely many locally closed subsets. Given $\mathfrak{p} \in F(\mathcal{X}^\dag)$, denote by $U_\mathfrak{p} = \Pi^{-1}(\{\mathfrak{p}\})$ the corresponding locally closed subset of $\mathcal{X}$, equipped with the reduced subscheme structure. Each locally closed stratum $U_\mathfrak{p}$ is smooth over $k$, and its Zariski closure $V_\mathfrak{p}$ (again with the reduced structure) is normal.

We denote by $\mathcal{M}_\mathcal{X} \hookrightarrow \mathcal{O}_{\mathcal{X}}$ the log structure on $\mathcal{X}^\dag$. Take $\mathfrak{p} \in F(\mathcal{X}^\dag)$. If $\mathfrak{p}$ is the generic point of $F(\mathcal{X}^\dag)$, we simply set $m^\sharp_\mathfrak{p} = 1$. Otherwise, we denote by $m^\sharp_\mathfrak{p}$ the largest positive integer such that the image of $1$ under the map $\mathbf{N} \to \mathcal{M}_{\mathcal{X},\mathfrak{p}}^\sharp$ associated with $\mathcal{X}^\dag \to S^\dag$ is divisible by $m^\sharp_\mathfrak{p}$.

\begin{defi} \label{defi:locus} Define $F_\pp(\mathcal{X}^\dag)$ (resp.\ $F_\tp(\mathcal{X}^\dag)$) to be the subset of $F(\mathcal{X}^\dag)$ consisting of those $\mathfrak{p}$ such that $m^\sharp_\mathfrak{p}$ is divisible (resp.\ is not divisible) by $p$. Let $$\mathcal{X}_\pp = \Pi^{-1}(F_\pp(\mathcal{X}^\dag)) \ \text{ and } \ \mathcal{X}_\tp = \Pi^{-1}(F_\tp(\mathcal{X}^\dag))\rlap{.}$$ We refer to $\mathcal{X}_\pp$ as the \emph{$p$-locus} of $\mathcal{X}$, and to $\mathcal{X}_\tp$ as the \emph{$p'$-locus} of $\mathcal{X}$. Put $m_{\mathcal{X}^\dag} = \gcd_{\mathfrak{p}}(m^\sharp_{\mathfrak{p}})$, where $\mathfrak{p}$ ranges over all elements of $F(\mathcal{X}^\dag)$ different from the generic point. \end{defi}

\begin{rema} The $p'$-locus $\mathcal{X}_\tp$ contains the generic fibre of $\mathcal{X} \to S$, and therefore, $\mathcal{X}_\pp$ is contained in the special fibre. Assume that $\mathfrak{p}$ and $\mathfrak{q}$ are points on $F(\mathcal{X}^\dag)$, not equal to the generic point, such that $\mathfrak{p}$ specialises to $\mathfrak{q}$. If $\mathfrak{q} \in F_\pp(\mathcal{X}^\dag)$, then also $\mathfrak{p} \in F_\pp(\mathcal{X}^\dag)$. This means precisely that the $p'$-locus $\mathcal{X}_\tp$ cuts out a closed subset of the special fibre. \end{rema}

As we will see shortly, $\mathcal{X}^\dag \to S^\dag$ is log smooth at all points of $\mathcal{X}_\tp$, but for points of $\mathcal{X}_\pp$, the situation is less clear; our goal is to pin down the subset of $\mathcal{X}_\pp$ where $\mathcal{X}^\dag \to S^\dag$ is log smooth using a criterion involving differential forms. This generalises (in several ways) an idea used in \cite[\S 5]{Sm}.

\begin{prop} \label{prop:omega} Let $\mathcal{X}$ be as above. Put $m = m_{\mathcal{X}^\dag}$. \begin{itemize} \item[(1)] For any $x\in\mathcal{X}_s$, there exist an affine neighbourhood $\mathrm{Spec}\,A$ of $x$, $u\in A^\times$ and $f\in A$ such that $uf^m$ is equal to the pullback of a uniformiser of $R$. \item[(2)] We denote by $\overline{\mathcal{X}}$ the closed subscheme of $\mathcal{X}$ defined by $m$. Then there exists a unique $\omega_{\overline{\mathcal{X}}/S}\in H^0(\overline{\mathcal{X}},\Omega^1_{\overline{\mathcal{X}}/S})$ such that, for any $x\in\mathcal{X}_s$, $W=\mathrm{Spec}\,A$, $u\in A^\times$ and $f\in A$ as in (1), the equality $\omega_{\overline{\mathcal{X}}/S}|_{\overline{\mathcal{X}}\cap W}=\mathrm{dlog}\,\overline{u}$ holds, where $\overline{u}$ is the pullback of $u$ under the closed immersion $\overline{\mathcal{X}}\cap W\to W$. If $R$ is of positive characteristic, then $\overline{\mathcal{X}}=\mathcal{X}$ and $\omega_{\overline{\mathcal{X}}/S}|_X=0$. \end{itemize} \end{prop}

\begin{proof} Let $x \in \mathcal{X}_s$. Choose a uniformiser $\pi$ of $R$. The map $\mathbf{N} \to R$ given by $1 \mapsto \pi$ yields a chart for the log structure on $S^\dag$. Consider a chart for $\mathcal{X}^\dag \to S^\dag$ around $x$, consisting of an affine neighbourhood $W = \mathrm{Spec}\,A$ of $x$, an fs monoid $P$ and homomorphisms $v\colon \mathbf{N} \to P$ and $\varphi\colon P \to A$ for which $\varphi(v(1)) = \pi$.

Let $\mathfrak{p} \in F(\mathcal{X}^\dag)$ such that $x \in U_\mathfrak{p}$. The inverse image under $\varphi$ of the prime ideal of $A$ which defines $x$ is a prime ideal $\mathfrak{p}'$ of $P$. The complement $P \setminus \mathfrak{p}'$ is a face, so $P/(P \setminus \mathfrak{p}')$ is a sharp fs monoid, which implies that the image of $1$ under $\mathbf{N} \to P \to P/(P \setminus \mathfrak{p}')$ is divisible by $m_{\mathfrak{p}}^\sharp$. Thus, there exists an element $c \in P \setminus \mathfrak{p}'$ such that $v(1) - c$ is divisible by $m$ in $P$. Therefore, $v(1) - c = ma$ for some $a \in P$. Put $u = \varphi(c)$ and $f = \varphi(a)$. Then $u\in A\cap\mathcal{O}_{\mathcal{X},x}^\times$ and $\pi = uf^m$.

Shrinking $W$, we obtain $W=\mathrm{Spec}\,A$, $u\in A^\times$ and $f\in A$ as described in (1). Let us show (2). Take any other $W'=\mathrm{Spec}\,A'$, $u'\in(A')^\times$ and $f'\in A'$ as in (1). We denote by $r\in(R/(m))^\times$ the image of $uf^m(u'(f')^m)^{-1}\in R^\times$ under the quotient homomorphism $R\to R/(m)$. Put $w=\overline{f^{-1}f'}\in H^0(\overline{\mathcal{X}}\cap W\cap W',\mathcal{O}_{\overline{\mathcal{X}}}^\times)$. Then $\mathrm{dlog}\,\overline{u}-\mathrm{dlog}\,\overline{u'}=\mathrm{dlog}\,r+\mathrm{dlog}\,w^m=m\,\mathrm{dlog}\,w=0$ on $\overline{\mathcal{X}}\cap W\cap W'$. If $R$ is of positive characteristic, then $\mathrm{dlog}\,u=\mathrm{dlog}\,\pi=0$ on $X\cap W$, which concludes the proof of (2). \end{proof}

\begin{defi} \label{defi:omega} Let $\mathcal{X}$ be as above. We denote by $\mathcal{U}$ the open subscheme $\mathcal{X}\setminus(\mathcal{X}_\tp\cap\mathcal{X}_s)$ of $\mathcal{X}$ and by $\overline{\mathcal{U}}$ the closed subscheme of $\mathcal{U}$ defined by $m_{\mathcal{X}^\dag\cap\mathcal{U}}$ (Definition \ref{defi:locus}). Take $\omega_{\overline{\mathcal{U}}/S}\in H^0(\overline{\mathcal{U}},\Omega^1_{\overline{\mathcal{U}}/S})$ given by Proposition \ref{prop:omega}.(2). Let $U$ be a subscheme of $\overline{\mathcal{U}}$ and $\mathcal{X}_s$. We denote by $$\begin{tikzcd}\lambda_U\colon H^0(\overline{\mathcal{U}},\Omega^1_{\overline{\mathcal{U}}/S})\ar[r]&H^0(U,\Omega^1_{U/S})\cong H^0(U,\Omega^1_{U/k})\end{tikzcd}$$ the composite of the homomorphism induced by the pullback under the immersion $U\to\overline{\mathcal{U}}$ and the isomorphism induced by the reduction homomorphism $R\to k$. Put $\omega_U=\lambda_U(\omega_{\overline{\mathcal{U}}/S})$. For each $\mathfrak{p} \in F_\pp(\mathcal{X}^\dag)$, put $\omega_\mathfrak{p}=\omega_{U_\mathfrak{p}}\in H^0(U_{\mathfrak{p}},\Omega^1_{U_{\mathfrak{p}}/k})$. \end{defi}

\begin{prop} \label{prop:logsmoothlocus} Let $\mathcal{X}$ be as above, and let $x \in \mathcal{X}$. \begin{itemize} \item[(1)] Assume that $x \in \mathcal{X}_\tp$. Then the morphism $\mathcal{X}^\dag \to S^\dag$ is log smooth at $x$. \item[(2)]  Assume that $x \in \mathcal{X}_\pp$. Let $\mathfrak{p} \in F_\pp(\mathcal{X}^\dag)$ such that $x \in U_\mathfrak{p}$. Then the morphism $\mathcal{X}^\dag \to S^\dag$ is log smooth at $x$ if and only if $\omega_\mathfrak{p}$ does not vanish at $x$. \end{itemize} \end{prop}

\begin{proof} We use the same notation as in the proof of Proposition \ref{prop:omega}. If $\mathfrak{p} \in F_\tp(\mathcal{X}^\dag)$, the cokernel of $v^\mathrm{gp}\colon\mathbf{N}^\mathrm{gp} \to P^\mathrm{gp}$ does not have $p$-torsion, simply because the cokernel of the composite $\mathbf{N}^\mathrm{gp} \to P^\mathrm{gp}/(P \setminus \mathfrak{p}')^\mathrm{gp}$ does not have any. Hence, Kato's criterion implies that the morphism $\mathcal{X}^\dag \to S^\dag$ is log smooth at $x$. This settles (1).

Let us prove (2). A slight refinement of Kato's criterion (see \cite[Theorem 12.3.37, Corollary 12.3.42]{GR}) says that the morphism $\mathcal{X}^\dag \to S^\dag$ is log smooth at $x$ if and only if one can choose, \'etale locally around $x$, a chart for $\mathcal{X}^\dag \to S^\dag$, consisting of an affine neighbourhood $W = \mathrm{Spec}\,A$, a toric monoid $P$ and morphisms $v\colon \mathbf{N} \to P$ and $\varphi\colon P \to A$ for which $\varphi(v(1)) = \pi$, such that the induced map $$\begin{tikzcd}\Phi\colon W \ar[r]& \mathrm{Spec}\,R[P]/(\pi - v(1))\end{tikzcd}$$ is \'etale, and the cokernel of $v^\mathrm{gp}\colon \mathbf{N}^\mathrm{gp} \to P^\mathrm{gp}$ does not have $p$-torsion.

Let us show the only if part. Assume that the morphism $\mathcal{X}^\dag \to S^\dag$ is log smooth at $x$. We choose the above chart and take $\mathfrak{p}'$ and $c \in P \setminus \mathfrak{p}'$ as in the proof of Proposition \ref{prop:omega}.(1), where $p$ divides $v(1) - c$. The ideal defining $V_\mathfrak{p}$ is $(\mathfrak{p}')$, the ideal generated by $\{\varphi(b) : b \in \mathfrak{p}'\}$. Shrinking $W$, we may assume that $V_\mathfrak{p} \cap W = U_\mathfrak{p} \cap W$, which is smooth over $k$. Since $\Phi$ is \'etale, so is the base change $$\begin{tikzcd}U_\mathfrak{p} \cap W \ar[r]& \mathrm{Spec}\,k[P \setminus \mathfrak{p}']\rlap{.}\end{tikzcd}$$ It follows that in this case, the $1$-forms on $U_{\mathfrak{p}} \cap W$ are given by $$\Omega^1_{U_\mathfrak{p}/k}(U_\mathfrak{p} \cap W) \cong A/(\mathfrak{p}') \otimes_{\mathbf{Z}} (P \setminus \mathfrak{p}')^\mathrm{gp}\rlap{,}$$ and the restriction of $\omega_\mathfrak{p}$ to $U_\mathfrak{p} \cap W$ corresponds to $\overline{1} \otimes c$. Since $\mathrm{coker}\,v^\mathrm{gp}$ does not have $p$-torsion, and $p$ divides $v(1) - c$, the element $c$ is not divisible by $p$ in $(P \setminus \mathfrak{p}')^\mathrm{gp}$. Thus, we have $$\overline{1} \otimes c \ne 0\ \text{ in }\ \kappa(x) \otimes_{\mathbf{Z}} (P \setminus \mathfrak{p}')^\mathrm{gp}\rlap{,}$$ where $\kappa(x)$ is the residue field of the local ring $\mathcal{O}_{\mathcal{X},x}$. Therefore, we conclude that $\omega_\mathfrak{p}$ does not vanish at $x$.

Let us show the if part. Assume that $\omega_\mathfrak{p}$ does not vanish at $x$. We may assume that $x$ is a closed point since both log smoothness of $\mathcal{X}^\dag \to S^\dag$ and non-vanishing of $\omega_\mathfrak{p}$ are open conditions. Replacing $S$ by an \'etale covering, we may assume that the residue field of $\mathcal{O}_{\mathcal{X},x}$ is equal to $k$. Put $e=\dim U_{\mathfrak{p}}$. Then $e>0$ since $\omega_\mathfrak{p}$ does not vanish at $x$. Put $P_0=\mathcal{M}_x/\mathcal{O}_{\mathcal{X},x}^\times$ and $P=P_0\oplus\mathbf{N}^e$. Since $\mathcal{X}^\dag$ is log regular, the equalities $\dim\mathcal{O}_{\mathcal{X},x}=e+\mathrm{rank}\,P_0^\mathrm{gp}=\mathrm{rank}\,P^\mathrm{gp}$ hold \cite[2.1]{Kato}, and we may choose a chart for $\mathcal{X}^\dag$ around $x$, consisting of an affine neighbourhood $W = \mathrm{Spec}\,A$ of $x$ and a morphism $\varphi_0\colon P_0\to A$ \cite[1.6]{Kato}. We take $u\in A\cap\mathcal{O}_{\mathcal{X},x}^\times$ and $c_0\in P_0$ so that $\pi = u\varphi_0(pc_0)$ as before. Shrinking $W$, we may take a lifting $(a_i\in A)_{i=1}^e$ of a system of parametners of the regular local ring $\mathcal{O}_{U_{\mathfrak{p}},x}$ so that $a_1-u\in R$ since $d\overline{u}=\overline{u}\,\mathrm{dlog}\,\overline{u}$ does not vanish at $x$. Shrinking $W$, we may assume that $u\in A^\times$ and $a_i-1\in A^\times$ for any $i$. Then the morphisms $v\colon \mathbf{N} \to P$, $n\mapsto(npc_0,n,0,\dots,0)$ and $\varphi\colon P \to A$, $(c,n_1,\dots,n_e)\mapsto \varphi_0(c)u^{n_1}\prod_{i=2}^e(a_i-1)^{n_i}$ yield a chart for $\mathcal{X}^\dag \to S^\dag$, where $\varphi(v(1)) = \pi$, and the cokernel of $v^\mathrm{gp}\colon \mathbf{N}^\mathrm{gp} \to P^\mathrm{gp}$ does not have $p$-torsion. Let us show that the induced map $$\begin{tikzcd}\Phi\colon W \ar[r]& W_0=\mathrm{Spec}\,R[P]/(\pi - v(1))\end{tikzcd}$$ is \'etale at $x$. The homomorphism $\Phi^\sharp\colon\widehat{\mathcal{O}}_{W_0,\Phi(x)}\to\widehat{\mathcal{O}}_{W,x}$ associated with $\Phi$ between the completions with respect to the maximal ideals is surjective. Thus, since $\dim\widehat{\mathcal{O}}_{W,x}=\mathrm{rank}\,P^\mathrm{gp}=\dim\widehat{\mathcal{O}}_{W_0,\Phi(x)}$, and both $\widehat{\mathcal{O}}_{W,x}$ and $\widehat{\mathcal{O}}_{W_0,\Phi(x)}$ are integral domains, the homomorphism $\Phi^\sharp$ is an isomorphism, which implies that $\Phi$ is \'etale at $x$. Therefore, the morphism $\mathcal{X}^\dag \to S^\dag$ is log smooth at $x$. \end{proof}

\begin{prop} \label{prop:vanish} We use the notation introduced in Definition \ref{defi:omega}. \begin{itemize} \item[(1)] Let $\mathfrak{p} \in F(\mathcal{X}^\dag)$ such that $V_\mathfrak{p}\subset\mathcal{X}_\pp$ and $\mathcal{X}^\dag\to S^\dag$ is log smooth on $V_\mathfrak{p}$. Then $\omega_{V_\mathfrak{p}}$ is a nowhere vanishing $1$-form. \item[(2)] Let $U$ be an open subscheme of the reduction of $\mathcal{X}_s$ contained in $\mathcal{X}_\pp$. Then the following statements are equivalent: (a) $\mathcal{X}^\dag\to S^\dag$ is nowhere log smooth on $U$; (b) $\omega_U$ vanishes on a dense open subset; (c) $\omega_U=0$. \end{itemize} \end{prop}
\begin{proof} Proposition \ref{prop:logsmoothlocus}.(2) shows (1) since $V_\mathfrak{p}$ is the disjoint union of subschemes $U_\mathfrak{q}$ of $V_\mathfrak{p}$, where $\mathfrak{q}$ ranges over all specializations of $\mathfrak{p}$ in $F_\pp(\mathcal{X}^\dag)$, and $\omega_{U_\mathfrak{q}}$ is equal to the pullback of $\omega_{V_\mathfrak{p}}$ under the immersion $U_\mathfrak{q}\to V_\mathfrak{p}$.

Let us show (2). The disjoint union of open subsets $U\cap U_\mathfrak{p}$ of $U$ is dense in $U$, where $\mathfrak{p}$ ranges over all points on $F_\pp(\mathcal{X}^\dag)$ of codimension $1$, and $\omega_{U_\mathfrak{p}}$ is equal to the pullback of $\omega_U$ under the open immersion $U_\mathfrak{p}\to U$. Thus, since the log smoothness of $\mathcal{X}^\dag\to S^\dag$ is an open condition, Proposition \ref{prop:logsmoothlocus}.(2) shows the equivalence of (a) and (b). Since (c) immediately implies (b), it remains to prove the converse.

Assume that (b) holds. We use the same notation as in the proof of Proposition \ref{prop:omega}. We may assume that $\mathcal{X}=W$ and $u\in A^\times$. Take $\mathfrak{p}\in F_\pp(\mathcal{X}^\dag)$ of codimension $1$. We denote by $u_\mathfrak{p}$ the pullback of $u$ under the immersion $U\cap V_\mathfrak{p}\to W$. Then $\mathrm{dlog}\,u_\mathfrak{p}$ vanishes on a dense open subset by assumption. Since $V_\mathfrak{p}$ is normal, Lemma \ref{lemm:root} gives $w_\mathfrak{p}\in H^0(U\cap V_\mathfrak{p},\mathcal{O}_{V_\mathfrak{p}}^\times)$ such that $w_\mathfrak{p}^p=u_\mathfrak{p}$. For any other $\mathfrak{q}$, the pullbacks of $w_\mathfrak{p}$ and $w_\mathfrak{q}$ to the reduced closed subscheme $U\cap V_\mathfrak{p}\cap V_\mathfrak{q}$ are equal. Thus, Lemma \ref{lemm:paste} shows that the sections $w_\mathfrak{p}$ for all $\mathfrak{p}$ glue to $w\in H^0(U,\mathcal{O}_U^\times)$, and the equality $w^p=u$ holds. Therefore, (c) holds, which concludes the proof.
\end{proof}

The following corollary illustrates the fact that log smoothness must show strongly restricted geometric behaviour on $\mathcal{X}_\pp$; for example, any maximally degenerate point contained in $\mathcal{X}_\pp$ cannot be a log smooth point:

\begin{coro} \label{coro:restrictions} Let $X$ as above, and let $\mathfrak{p} \in F(\mathcal{X}^\dag)$ such that $V_\mathfrak{p}\subset\mathcal{X}_\pp$ and $\mathcal{X}^\dag\to S^\dag$ is log smooth on $V_\mathfrak{p}$. Then $V_\mathfrak{p}$ cannot be zero-dimensional. If $V_\mathfrak{p}$ is one-dimensional and proper over $k$, then it is a curve of genus $1$. If $V_\mathfrak{p}$ is two-dimensional, smooth and proper over $k$, then it is not of general type. \end{coro}

\begin{proof} Proposition \ref{prop:vanish}.(1) shows that $\Omega^1_{V_\mathfrak{p}/k}$ has a nowhere vanishing global section. Thus, the first two statements hold, and the last statement follows from \cite[Theorem 5.1]{La}. \end{proof}

\section{Existence of tame points} \label{section3}

We keep the notation from the previous section. We will present sufficient geometric conditions for the existence of a $K^t$-point on smooth, proper $K$-varieties with logarithmic good reduction (in particular, we give a proof of Theorem \ref{maintheorem}). We need the following fact:

\begin{lemm} \label{lemm:index} Let $X$ be a smooth, proper $K$-variety, and let $\mathcal{X}$ be a proper model of $X$ over $S$ such that $\mathcal{X}^\dag$ is log regular. Then $X(K^t) \neq \emptyset$ if and only if $\mathcal{X}_\pp \neq \mathcal{X}_s$. \end{lemm}

\begin{rema} This result does not seem to be available in the literature in this precise form. If $\mathcal{X}$ is (classically) regular, then it follows immediately from \cite[Corollary 2.7]{KN} or the much more general \cite[Theorem 8.2]{GLL}. To reduce to this case, it suffices to argue that the equality $\mathcal{X}_s = \mathcal{X}_\pp$ remains invariant under log blow-ups, since any (potentially singular) log regular model can be transformed into a classically regular model using such transformations.

Recall that $\mathcal{X} \to S$ induces a morphism of fans $F(\mathcal{X}^\dag) \to F(S^\dag)$. We have a canonical identification $F(S^\dag) \cong \mathrm{Spec}\,\mathbf{N}$. Saying that $\mathcal{X}_\pp = \mathcal{X}_s$ is equivalent to saying that for \emph{every} morphism of fans $\mathrm{Spec}\,\mathbf{N} \to F(\mathcal{X}^\dag)$, the composite $\mathrm{Spec}\,\mathbf{N} \to F(\mathcal{X}^\dag) \to F(S^\dag) \cong \mathrm{Spec}\,\mathbf{N}$ corresponds to a homomorphism of monoids $\mathbf{N} \to \mathbf{N}$ which maps $1$ to a multiple of $p$. On the other hand, if $\widetilde{\mathcal{X}}^\dag \to \mathcal{X}^\dag$ is a log blow-up, the induced map $F(\widetilde{\mathcal{X}}^\dag)(\mathbf{N}) \to F(\mathcal{X})(\mathbf{N})$ is bijective by \cite[\S 9]{Kato}. Therefore, we have $\mathcal{X}_\pp = \mathcal{X}_s$ if and only if $\widetilde{\mathcal{X}}_\pp = \widetilde{\mathcal{X}}_s$, as desired. \end{rema}

The following result generalises \cite[Proposition 5.1]{Sm}.

\begin{prop} \label{prop:vanishingofchi} Let $X$ be a smooth, proper $K$-variety, and let $\mathcal{X}$ be a proper, regular model of $X$ over $S$ such that $\mathcal{X}^\dag$ is log smooth over $S^\dag$. If $\mathfrak{p} \in F(\mathcal{X}^\dag)$ satisfies $V_\mathfrak{p}\subset\mathcal{X}_\pp$, then $\chi(V_\mathfrak{p}) = \chi(U_\mathfrak{p}) = 0$. In particular, if $X(K^t) = \emptyset$, or equivalently, $\mathcal{X}_\pp=\mathcal{X}_s$, then $\chi(\mathcal{X}_s) = 0$. \end{prop}

\begin{proof} Since $V_\mathfrak{p}$ is smooth and proper over $k$, the Euler number $\chi(V_\mathfrak{p})$ -- by which we mean the $\ell$-adic Euler characteristic, for any prime $\ell \neq p$ -- can be computed as the positive or negative degree of the top Chern class of the vector bundle $\Omega^1_{V_\mathfrak{p}/k}$ on $V_\mathfrak{p}$; the sign depends on the dimension. Proposition \ref{prop:vanish}.(1) implies that $\Omega^1_{V_\mathfrak{p}/k}$ has a nowhere vanishing global section, and hence, its top Chern class vanishes, which implies that $\chi(V_\mathfrak{p})=0$. Thus, the first statement follows by the induction on the dimension of $V_\mathfrak{p}$ since $V_\mathfrak{p}$ is the disjoint union of $U_\mathfrak{q}$, where $\mathfrak{q}$ ranges over all specializations of $\mathfrak{p}$ in $F_\pp(\mathcal{X}^\dag)$ (including $\mathfrak{p}$). The last statement follows from the first one since $\mathcal{X}_\pp$ is the disjoint union of $U_\mathfrak{q}$, where $\mathfrak{q}$ ranges over all points on $F_\pp(\mathcal{X}^\dag)$. \end{proof}

We are now able to prove Theorem \ref{maintheorem}:

\begin{proof}[Proof of Theorem \ref{maintheorem}] The assumptions of the theorem and \cite[Corollary 0.1.1]{Na} together imply that the wild inertia of $\mathrm{Gal}(K^s/K)$ acts trivially on $H^\star(X \times_K K^s, \mathbf{Q}_\ell)$. Hence, \begin{equation*}  \chi(X) = \chi^\mathrm{tame}(X) := \sum_{i \geq 0} (-1)^i \dim_{\mathbf{Q}_\ell} H^i(X \times_K K^t, \mathbf{Q}_\ell)\rlap{.}\end{equation*} 

The right hand side can be computed by means of Nakayama's description of nearby cycles for log smooth families \cite[Theorem 3.5]{Na}. Computations of this type have been carried out by Nicaise in \cite{Ni} and also by the second named author in \cite[\S 3]{Sm}, where the \emph{tame monodromy zeta function} $\zeta_X^\mathrm{tame}$ was calculated. Since $k$ is algebraically closed, the group $\mathrm{Gal}(K^t/K)$ is procyclic, topologically generated by an element $\varphi$. The tame monodromy zeta function is given by $$\zeta_X^\mathrm{tame}(t) := \prod_{m \geq 0} \det\left(t \cdot \mathrm{Id} - \varphi \mid H^m(X \times_K K^t, \mathbf{Q}_\ell)\right)^{(-1)^{m + 1}}\rlap{,}$$ which is independent of the choice of $\varphi$ (as will be clear from (\ref{eq:tamezeta}) below). Given a point $\mathfrak{p}\in F(\mathcal{X}^\dag)$ of codimension $1$, denote by $m_\mathfrak{p}'$ the biggest prime-to-$p$ divisor of the multiplicity $m_\mathfrak{p}$ of the corresponding component of $\mathcal{X}_s$. Then \cite[Corollary 3.8]{Sm} yields

\begin{equation} \label{eq:tamezeta} \zeta_X^\mathrm{tame}(t) = \prod_{\mathfrak{p}} (t^{m_\mathfrak{p}'} - 1)^{-\chi(U_\mathfrak{p})}\rlap{,}\end{equation} where $\mathfrak{p}$ ranges over all points on $F(\mathcal{X}^\dag)$ of codimension $1$. Hence, we obtain \begin{equation} \label{eq:tameeulerbis} \chi^\mathrm{tame}(X) = \sum_{\mathfrak{p}} m'_\mathfrak{p} \chi(U_\mathfrak{p})\rlap{,}\end{equation}
simply by taking the negative degrees on both sides of (\ref{eq:tamezeta}). We may assume that $\mathcal{X}$ is regular after desingularization by log blow-ups. Thus, if $X(K^t) = \emptyset$, then Proposition \ref{prop:vanishingofchi} says that the right hand side of (\ref{eq:tameeulerbis}), and hence also $\chi(X)$, vanish. This proves the theorem.
\end{proof}

\section{A numerical criterion} \label{section4}

As in the previous paragraphs, we assume that $X$ is a smooth, proper $K$-variety and that $\mathcal{X}$ is a proper, flat model of $X$ over $S$ such that $\mathcal{X}^\dag$ is log regular.

The special fibre $\mathcal{X}_s$ considered as a Weil divisor on $\mathcal{X}$ may be written as $$\mathcal{X}_s = \textstyle \sum_{\mathfrak{p} \in F(\mathcal{X}^\dag)^{(1)}} m_\mathfrak{p}^\sharp V_\mathfrak{p}\rlap{,}$$ where $F(\mathcal{X}^\dag)^{(1)}$ denotes the set of points on $F(\mathcal{X}^\dag)$ of codimension $1$. Put $m = m_{\mathcal{X}^\dag}$ (Definition \ref{defi:locus}) and $n_\mathfrak{p} = m_\mathfrak{p}^\sharp/m$ for each $\mathfrak{p} \in F(\mathcal{X}^\dag)$. We define a Weil divisor on $\mathcal{X}$ by $$D = \textstyle \sum_{\mathfrak{p} \in F(\mathcal{X}^\dag)^{(1)}} n_\mathfrak{p} V_\mathfrak{p}\rlap{.}$$ Then $D$ is a Cartier divisor by Proposition \ref{prop:omega}.(1).

We denote by $E$ the reduction of the special fibre $\mathcal{X}_s$ of $\mathcal{X}$, and by $\mathcal{L}$ the pullback of $\mathcal{O}_{\mathcal{X}}(D)$ under the closed immersion $E\to\mathcal{X}$. As an element of $\mathrm{Pic}\,E$, the line bundle $\mathcal{L}$ has finite order dividing $m$. Indeed, $\mathcal{L}^{\otimes m}$ is the trivial line bundle since $mD = \mathcal{X}_s$ and $\mathcal{O}_{\mathcal{X}}(\mathcal{X}_s)\cong\mathcal{O}_{\mathcal{X}}$. Let us denote by $\mu$ the order of $\mathcal{L}$.

\begin{lemm} \label{lemm:power} The ratio $m/\mu$ is a power of $p$. \end{lemm}

\begin{proof} Let $$\mathcal{I} =\mathcal{O}_{\mathcal{X}}(-D),\ \mathcal{I}_0=\sqrt{\mathcal{I}},\ \mathcal{I}_n=\mathcal{I}^n\ \textrm{ and }\ \mathcal{N}_n =\mathcal{I}_n/\mathcal{I}_{n+1}$$ for each $n\in\mathbf{Z}_{>0}$. For each $n\in\mathbf{Z}_{\geq0}$, we denote by $i_n\colon\mathcal{Y}_n\to\mathcal{X}$ the closed immersion defined by $\mathcal{I}_n$, and by $\mu_n$ the order of $i_n^*\mathcal{O}_{\mathcal{X}}(D)$ in $\mathrm{Pic}\,\mathcal{Y}_n$. Then $\mathcal{Y}_0=E$, and $\mu_0=\mu$. Take $n\in\mathbf{Z}_{\geq0}$. Then the exact sequence of sheaves of abelian groups on $\mathcal{X}$ given by $$\begin{tikzcd}0\ar[r]&1+\mathcal{N}_n\ar[r]&(i_{n+1})_*\mathcal{O}_{\mathcal{Y}_{n+1}}^\times\ar[r]&(i_n)_*\mathcal{O}_{\mathcal{Y}_n}^\times\ar[r]&0\end{tikzcd}$$ yields the exact sequence of abelian groups $$\begin{tikzcd}H^1(\mathcal{X},1+\mathcal{N}_n)\ar[r]&\mathrm{Pic}\,\mathcal{Y}_{n+1}\ar[r,"\alpha"]&\mathrm{Pic}\,\mathcal{Y}_n\rlap{.}\end{tikzcd}$$ Since $1+\mathcal{N}_n$ is $p$-torsion, the abelian group $H^1(\mathcal{X},1+\mathcal{N}_n)$ is $p$-torsion as well. It follows that the kernel of $\alpha$ is $p$-torsion, which implies that $\mu_{n+1}\mid p\mu_n$. If $r \in \mathbf{Z}$ is sufficiently large, then $\mu_r = m$ by \cite[Lemme 6.4.4]{Ra}. Since $\mu_0=\mu$, the result follows.
\end{proof}

\begin{prop} \label{prop:normalbundle} With notation as above, the following conditions are equivalent: \begin{itemize} \item[(1)] there exists a point $x$ on $\mathcal{X}_s$ such that $\mathcal{X}^\dag \to S^\dag$ is log smooth at $x$; \item[(2)] the equality $m=\mu$ holds. \end{itemize} \end{prop}

\begin{proof} If $m$ is prime to $p$, then (1) holds by Proposition \ref{prop:logsmoothlocus}.(1), and (2) holds by Lemma \ref{lemm:power}. Thus, we may assume that $p\mid m$. Put $\mathcal{L}'=\mathcal{L}^{\otimes(m/p)}$. Lemma \ref{lemm:power} implies that (2) is equivalent to the condition that $\mathcal{L}'$ is non-trivial. Take $\omega_E\in H^0(E,\Omega^1_{E/k})$ in Definition \ref{defi:omega}. Proposition \ref{prop:vanish}.(2) implies that (1) is equivalent to the condition $\omega_E\not=0$. Thus, Lemma \ref{lemm:exact} concoludes the proof since $\psi(\omega_E)=\mathcal{L}'$ by Remark \ref{rema:explicit}. \end{proof}

\section{Logarithmic good reduction of curves of genus $1$} \label{section5}

In this section, we will prove Theorem \ref{genus1theorem}. For simplicity, we will assume the residue field $k$ to be algebraically closed (as we did in the introduction). Recall that if $X$ is a curve of positive genus over $K$, then among all proper, regular models of $X$ over $S$ with strict normal crossings special fibre, there is a minimal one, the \emph{minimal sncd model}. 

Let us first show that to verify whether a given curve of positive genus over $K$ has logarithmic good reduction, it suffices to consider this particular model.

\begin{lemm} \label{lemm:minimal} Let $X$ be a smooth, proper curve of positive genus over $K$. If $X$ has logarithmic good reduction and if $\mathcal{X}$ is its minimal sncd model, then $\mathcal{X}^\dag \to S^\dag$ is log smooth. \end{lemm}

\begin{proof} Let $\mathcal{X}_1$ be a proper model of $X$ over $S$ such that $\mathcal{X}_1^\dag \to S^\dag$ is log smooth. There exists a desingularisation by log blow-ups $\mathcal{X}_2^\dag \to \mathcal{X}_1^\dag$, so that the underlying scheme $\mathcal{X}_2$ is regular. Then $\mathcal{X}_2^\dag \to S^\dag$ is still log smooth since log blow-ups are log \'etale.

By minimality, we get a morphism $\mathcal{X}_2 \to \mathcal{X}$ to the minimal sncd model. This morphism successively contracts $(-1)$-curves which intersect at most two other components  of the special fibre, each of them in one point. If such a $(-1)$-curve meets two other components, the corresponding contraction is a log blow-up and preserves log smoothness. If it meets only one component, the contraction map is no longer a log blow-up. However, the multiplicities of the $(-1)$-curve and of the component which it intersects must be prime to $p$: if not, the original morphism would not be log smooth at the intersection point by Corollary \ref{coro:restrictions}. Hence, the contraction again preserves log smoothness and $\mathcal{X}^\dag \to S^\dag$ is log smooth, as required. \end{proof}

Let us first deal with the simpler cases of Theorem \ref{genus1theorem}.

\begin{lemm} Let $C$ be a smooth, proper curve of genus $1$ over $K$ with period prime to $p$. If $H^1(C \times_K K^s,\mathbf{Q}_\ell)$ is tamely ramified, then $C$ has logarithmic good reduction. \end{lemm}

\begin{proof} We denote by $m$ the period of $C$ and by $J$ its Jacobian. Let $\mathcal{J}$ (resp.\ $\mathcal{C}$) be the minimal sncd model of $J$ (resp.\ $C$) over $S$. The existence of a $\mathrm{Gal}(K^s/K)$-equivariant isomorphism $$H^1(C \times_K K^s,\mathbf{Q}_\ell) \cong H^1(J \times_K K^s,\mathbf{Q}_\ell)$$ shows that $H^1(J \times_K K^s,\mathbf{Q}_\ell)$ is also tamely ramified. Since $J$ is an elliptic curve, the special fibre $\mathcal{J}_s$ of $\mathcal{J}$ satisfies Saito's criterion \cite[Theorem 3.11]{Sa}, i.e.\ each component with multiplicity divisible by $p$ is a copy of $\mathbf{P}^1_k$ intersecting exactly two other components, with multiplicities prime to $p$. Hence, the special fibre $\mathcal{C}_s$ of $\mathcal{C}$ satisfies the same property: indeed, it is a consequence of \cite[Theorem 6.6]{LLR} that the weighted dual graph of $\mathcal{C}_s$ is obtained by multiplying the weighted dual graph for $\mathcal{J}_s$ by the integer $m$, which is prime to $p$ (see Remark \ref{rema:sncdtype} below). This yields the result. \end{proof}

\begin{rema} \label{rema:sncdtype} We should note that \cite[Theorem 6.6]{LLR} deals with the \emph{types} associated with the minimal regular models instead of the minimal sncd models considered here; however, it is not hard to see that the conclusion of [\emph{loc.\,cit.}] remains valid for the minimal sncd models and the weighted dual graphs of their special fibres. \end{rema}

Hence, the only remaining cases of Theorem \ref{genus1theorem} are those where $p\mid m$. The case where the Jacobian has bad reduction is not difficult either:

\begin{lemm} Let $C$ be a smooth, proper curve of genus $1$ over $K$ with period divisible by $p$. If its Jacobian $J$ has bad reduction, then $C$ does not have logarithmic good reduction. \end{lemm}

\begin{proof} Let $\mathcal{J}$ (resp.\ $\mathcal{C}$) be the minimal sncd model of $J$ (resp.\ of $C$). Lemma \ref{lemm:minimal} implies that it suffices to check that the morphism $\mathcal{C}^\dag \to S^\dag$ is not log smooth.  This follows from the combination of \cite[Theorem 6.6]{LLR} (applied to the minimal sncd model, see Remark \ref{rema:sncdtype})  and Corollary \ref{coro:restrictions}. Indeed, the morphism is certainly not log smooth at any point where two components of $\mathcal{C}_s$ intersect. \end{proof}

We can now finish the proof of Theorem \ref{genus1theorem} as follows.

\begin{proof}[End of Proof of Theorem \ref{genus1theorem}] The only case which remains is the one where $p\mid m$ and the Jacobian $J$ of $C$ has good reduction. Again considering the minimal sncd models of $J$ and $C$, \cite[Proposition 8.1]{LLR} says that the $S$-group scheme $\mathcal{J}$ acts on $\mathcal{C}$, inducing a transitive action of $\mathcal{J}(k)$ on $\mathcal{C}(k)$. It follows that on the special fibre, the morphism $\mathcal{C}^\dag \to S^\dag$ is either everywhere log smooth, or nowhere at all.

The reduction $D$ of the special fibre $\mathcal{C}_s$ of $\mathcal{C}$ is isomorphic to an elliptic curve. Denote by $\mu$ the order of $\mathcal{N}_{D/\mathcal{C}}$ in $\mathrm{Pic}\,D$. Using the above observation and Proposition \ref{prop:normalbundle}, we see that $\mathcal{C}^\dag \to S^\dag$ is log smooth (resp.\ nowhere log smooth on $\mathcal{C}_s$) if $m = \mu$ (resp.\ if $m \neq \mu$). However, $m = \mu$ is equivalent to $\mathcal{C}$ being cohomologically flat over $S$ by \cite[Corollary 2.3.3]{BT}. \end{proof}

Let us close with some simple remarks.

\begin{rema} The condition $m = \mu$ forces the curve $D$ in the above proof to be ordinary: $\mathcal{N}^{\otimes \mu/p}_{D/X}$ yields a non-zero $p$-torsion of $\mathrm{Pic}^0\,D$, which cannot exist if $D$ is supersingular. \end{rema}

\begin{rema} A logarithmic version of the N\'eron--Ogg--Shafarevich criterion for good reduction of abelian varieties has been proven in \cite{BS}. It would be very interesting to have a full characterisation of torsors under abelian varieties with logarithmic good reduction. \end{rema}

\section{Appendix}
\begin{lemm} \label{lemm:root}
 Let $k$ be a perfect field of characteristic $p>0$, and let $Z$ be a normal integral $k$-scheme.
 Then the sequence of sheaves of abelian groups
 \[
  \begin{tikzcd}
   0\ar[r]&\mathcal{O}_Z\ar[r,"p"]&\mathcal{O}_Z\ar[r,"d"]&\Omega_{Z/k}^1
  \end{tikzcd}
 \]
 is exact, where $p$ is the power-to-$p$ homomorphism, and $d$ is the differential operator.
\end{lemm}
\begin{proof}
 For a $k$-algebra $A$, we denote by $p_A\colon A\to A$ the power-to-$p$ homomorphism and by $d_A\colon A\to\Omega_{A/k}$ the differential operator.
 Let $R$ be a normal integral $k$-algebra with fraction field $K$.
 Then $\mathrm{Im}\,p_R\subset\mathrm{Ker}\,d_R$.
 Thus, we have only to show that $\mathrm{Ker}\,d_R\subset\mathrm{Im}\,p_R$.
 The equality $d_R\otimes_RK=d_K$ implies that $\mathrm{Ker}\,d_R\subset\mathrm{Ker}\,d_K\cap R$.
 Since $R$ is normal, the equality $\mathrm{Im}\,p_R=\mathrm{Im}\,p_K\cap R$ holds.
 Since $k$ is perfect, the equality $\mathrm{Ker}\,d_K=\mathrm{Im}\,p_K$ holds \cite[Ch.\ 0, 21.4.6]{EGA4-1}, which concludes that $\mathrm{Ker}\,d_R\subset\mathrm{Im}\,p_R$.
\end{proof}
\begin{lemm} \label{lemm:monoid}
 Let $P$ and $Q$ be finitely generated sharp integral monoids.
 We denote by $\{P_\lambda\}_{\lambda\in\Lambda}$ the set of facets of $P$ \cite[I.2.3.7.(1)]{Og}.
 Put $\Gamma=\{(\lambda,\mu)\in\Lambda^2\mid\lambda\not=\mu\}$ and $P_{\lambda\mu}=P_\lambda\cap P_\mu$ for each $(\lambda,\mu)\in\Gamma$.
 We consider one of the following cases.
 \begin{itemize}
  \item[(1)]
  Let $\kappa$ be a field. Put $A=\kappa[[P,Q]]$.
  \item[(2)]
  Let $W$ be a complete valuation ring with residue field $\kappa$ of characteristic $p$ in which $p$ is a uniformiser.
  Choose $\theta\in W[[P,Q]]$ such that $\theta\equiv p\bmod(P\setminus\{1\},Q\setminus\{1\})$.
  Put $A=W[[P,Q]]/(\theta)$.
 \end{itemize}
 Put $A_\lambda=A/(P\setminus P_\lambda)$ for each $\lambda\in\Lambda$ and $A_{\lambda\mu}=A/(P\setminus P_{\lambda\mu})$ for each $(\lambda,\mu)\in\Gamma$.
 Then the sequence of $A$-modules
 \[
  \begin{tikzcd}
   A\ar[r]&\prod_{\lambda\in\Lambda}A_\lambda\ar[r]&\prod_{(\lambda,\mu)\in\Gamma}A_{\lambda\mu}
  \end{tikzcd}
 \]
 is exact, and the kernel of the first arrow is generated by $P\setminus\bigcap_{(\lambda,\mu)\in\Gamma}P_{\lambda\mu}$.
\end{lemm}
\begin{proof}
 In Case (1), we put $\kappa_0=\kappa$.
 In Case (2), we choose a complete system of representatives $\kappa_0\subset W$ of $\kappa$ such that $0\in\kappa_0$.
 Put $B=\{(a_{\alpha\beta}\in\kappa_0)_{(\alpha,\beta)\in P\times Q}\}$, $B_\lambda=\{(a_{\alpha\beta}\in\kappa_0)_{(\alpha,\beta)\in P_\lambda\times Q}\}$ for each $\lambda\in\Lambda$, and $B_{\lambda\mu}=\{(a_{\alpha\beta}\in\kappa_0)_{(\alpha,\beta)\in P_{\lambda\mu}\times Q}\}$ for each $(\lambda,\mu)\in\Gamma$.
 We give structures of pointed sets by equipping $A$, $A_\lambda$, $A_{\lambda\mu}$, $B$, $B_\lambda$ and $B_{\lambda\mu}$ with $0$ as base points.
 We define morphisms $\phi\colon B\to A$, $\phi_\lambda\colon B_\lambda\to A_\lambda$ and $\phi_{\lambda\mu}\colon B_{\lambda\mu}\to A_{\lambda\mu}$ of pointed sets by associating $(a_{\alpha\beta})$ with $\sum_{(\alpha,\beta)\in P\times Q}a_{\alpha\beta}\alpha\beta$ in $A$, $A_\lambda$ and $A_{\lambda\mu}$, respectively.
 Then all $\phi$, $\phi_\lambda$ and $\phi_{\lambda\mu}$ are bijective, and any square in the diagram of pointed sets
 \[
  \begin{tikzcd}
   B\ar[r]\ar[d,"\phi","\cong"']&\prod_{\lambda\in\Lambda}B_\lambda\ar[r]\ar[d,"(\phi_\lambda)_{\lambda\in\Lambda}","\cong"']&\prod_{(\lambda,\mu)\in\Gamma}B_{\lambda\mu}\ar[d,"(\phi_{\lambda\mu})_{(\lambda,\mu)\in\Gamma}","\cong"']\\
   A\ar[r]&\prod_{\lambda\in\Lambda}A_\lambda\ar[r]&\prod_{(\lambda,\mu)\in\Gamma}A_{\lambda\mu}
  \end{tikzcd}
 \]
 is commutative, where all horizontal arrow are the canonical projections.
 Since the upper sequence is exact, the lower sequence is exact, which concludes the proof.
\end{proof}
We apply the above lemma in the case where $P$ is saturated, and $Q=\mathbf{N}^r$.
By \cite[3.2]{Kato}, we obtain the following:
\begin{lemm} \label{lemm:paste}
 Let $(\mathcal{X},\mathcal{M}_{\mathcal{X}})$ be a log regular integral log scheme with fan $F$.
 We denote by $(Z_\lambda)_{\lambda\in\Lambda}$ the set of prime divisors on $\mathcal{X}$ corresponding to points on $F$ of codimension $1$.
 Put $\Gamma=\{(\lambda,\mu)\in\Lambda^2\mid\lambda\not=\mu\}$.
 We define reduced schemes by $Y=\bigcup_{\lambda\in\Lambda}Z_\lambda$ and $Z_{\lambda\mu}=Z_\lambda\cap Z_\mu$ for each $(\lambda,\mu)\in\Gamma$.
 Then the sequence of $\mathcal{O}_Y$-modules
 \[
  \begin{tikzcd}
   0\ar[r]&\mathcal{O}_Y\ar[r]&\prod_{\lambda\in\Lambda}(i_\lambda)_*\mathcal{O}_{Z_\lambda}\ar[r]&\prod_{(\lambda,\mu)\in\Gamma}(i_{\lambda\mu})_*\mathcal{O}_{Z_{\lambda\mu}}
  \end{tikzcd}
 \]
 is exact, where $i_\lambda\colon Z_\lambda\to Y$ and $i_{\lambda\mu}\colon Z_{\lambda\mu}\to Y$ are the closed immersions, and the arrows are induced by the closed immersions $Z_{\lambda\mu}\to Z_\lambda\to Y$ for all $(\lambda,\mu)\in\Gamma$.
\end{lemm}
All $Z_\lambda$ and $Z_{\lambda\mu}$ in the above lemma are normal.
Thus, Lemma \ref{lemm:root} shows the following:
\begin{lemm} \label{lemm:exact}
 We use the notation introduced in Lemma \ref{lemm:paste}.
 Let $k$ be a perfect field of characteristic $p>0$.
 Assume that $Y$ is a $k$-scheme.
 Then the sequence of sheaves of abelian groups
 \[
  \begin{tikzcd}
   0\ar[r]&\mathcal{O}_Y\ar[r,"p"]&\mathcal{O}_Y\ar[r,"d"]&\Omega_{Y/k}^1
  \end{tikzcd}
 \]
 is exact, where $p$ is the power-to-$p$ homomorphism, and $d$ is the differential operator.
 In particular, the sequence of sheaves of abelian groups
 \[
  \begin{tikzcd}
   1\ar[r]&\mathcal{O}_Y^\times\ar[r,"p"]&\mathcal{O}_Y^\times\ar[r,"\mathrm{dlog}"]&\mathrm{dlog}\,\mathcal{O}_Y^\times\ar[r]&0
  \end{tikzcd}
 \]
 is exact, where $\mathrm{dlog}\colon\mathcal{O}_Y^\times\to\Omega_{Y/k}^1$ is the log differential operator.
 Suppose that $H^0(Y,\mathcal{O}_Y^\times)$ is $p$-dividible, e.g.\ $H^0(Y,\mathcal{O}_Y)=k$.
 Then the connecting homomorphism induced by the above exact sequence gives an isomorphism $\psi\colon H^0(Y,\mathrm{dlog}\,\mathcal{O}_Y^\times)\cong(\mathrm{Pic}\, Y)[p]$.
\end{lemm}
\begin{rema} \label{rema:explicit}
 Let us give an explicit definition of $\psi$.
 Take $\omega\in H^0(Y,\mathrm{dlog}\,\mathcal{O}_Y^\times)$.
 We may take an open covering $(Y_i)_{i\in I}$ of $Y$ and $(u_i\in H^0(Y_i,\mathcal{O}_Y^\times))_{i\in I}$ such that $\omega|_{Y_i}=\mathrm{dlog}\,u_i$ for all $i\in I$.
 For each $(i,j)\in I^2$, we put $Y_{ij}=Y_i\cap Y_j$ and $u_{ij}=(u_i|_{Y_{ij}})(u_j|_{Y_{ij}})^{-1}\in H^0(Y_{ij},\mathcal{O}_Y^\times)$.
 Take $w_{ij}\in H^0(Y_{ij},\mathcal{O}_Y^\times)$ so that $u_{ij}=w_{ij}^p$.
 Then $\psi(\omega)$ is defined as the line bundle on $Y$ whose transition functions are given by $(w_{ij})_{(i,j)\in I^2}$.
\end{rema}
We apply the above lemma in Proposition \ref{prop:normalbundle}.



\subsection*{Acknowledgements} We gratefully acknowledge the support of FWO Vlaanderen\footnote{Financial support through a personal fellowship.}, JSPS\footnote{JSPS KAKENHI Grant Numbers JP25800018, JP24224001, JP17K14167, JP17H02832, JP17H06127, 17H02835, the JSPS Program for Advancing Strategic International Networks to Accelerate the Circulation of Talented Researchers based on OCAMI (Osaka City University Advanced Mathematical Institute), the Japan-France Research Cooperative Program.}, CNRS\footnote{The Japan-France Research Cooperative Program.}, the Max Planck Institute for Mathematics for its hospitality and the European Research Council\footnote{ERC Starting Grant MOTZETA (project 306610).}. We are also very grateful to Giulia Battiston, Ben Moonen, Johannes Nicaise, Takeshi Saito and Jakob Stix for useful discussions on this note. The special case of Theorem \ref{genus1theorem} where $p\mid m$ and the Jacobian of $C$ has good reduction has been covered independently by R\'emi Lodh. We thank him for making us aware of his results and for subsequent discussions.

\normalsize


\normalsize

\begin{thebibliography}{xx}

\bibitem{BS} 
{A. Bellardini \& A. Smeets, \emph{Logarithmic good reduction of abelian varieties}. Mathematische Annalen 369 (2017), 1435--1442.}

\bibitem{BT} 
{A. Bertapelle \& J. Tong, \emph{On torsors under elliptic curves and Serre's pro-algebraic structures}. Mathematische Zeitschrift 277 (2014), 91--147.}

\bibitem{EGA4-1}
{J. Dieudonn{\'e} \& A. Grothendieck, \emph{\'{E}l\'ements de g\'eom\'etrie alg\'ebrique. {IV}. \'{E}tude locale des sch\'emas et des morphismes de sch\'emas ({P}remi\`ere partie)}. Institut des Hautes \'Etudes Scientifiques. Publications Math\'ematiques 20 (1964), 5--259}

\bibitem{ELW}
{H. Esnault, M. Levine \& O. Wittenberg, \emph{Index of varieties over Henselian fields and Euler characteristic of coherent sheaves}. Journal of Algebraic Geometry 24 (2015), 693--718.}

\bibitem{GLL}
{O. Gabber, Q. Liu \& D. Lorenzini, \emph{The index of an algebraic variety}. Inventiones Mathematicae 192 (2013), 567--626.}
\bibitem{GR}
{O. Gabber \& L. Ramero, \emph{Foundations for almost ring theory}. arxiv.org/abs/math/0409584}
\bibitem{HL}
{B. Harbourne \& W. Lang, \emph{Multiple fibres on rational elliptic surfaces}. Transactions of the American Mathematical Society 307 (1988), 205--223.} 
\bibitem{Kato}
{K. Kato, \emph{Toric singularities}. American Journal of Mathematics, 116 (1994), 1073--1099.}
\bibitem{KU}
{T. Katsura \& K. Ueno, \emph{On elliptic surfaces in characteristic $p$}. Mathematische Annalen 272 (1985), 291--330.}
\bibitem{KN}
{L. Kesteloot \& J. Nicaise, \emph{The specialization index of a variety over a discretely valued field}. Proceedings of the American Mathematical Society 145 (2017), 585--599.}
\bibitem{La}
{A. Langer, \emph{Generic positivity and foliations in positive characteristic}. Advances in Mathematics 277 (2015), 1--23.}
\bibitem{Lichtenbaum}
{S. Lichtenbaum, \emph{The period-index problem for elliptic curves}. American Journal of Mathematics 90 (1968), 1209--1223.}
\bibitem{LLR}
{Q. Liu, D. Lorenzini \& M. Raynaud, \emph{N\'eron models, Lie algebras, and reduction of curves of genus one}. Inventiones Mathematicae 157 (2004), 455--518.}
\bibitem{Na} C. Nakayama, \emph{Nearby cycles for log smooth families}. Compositio Mathematica 112 (1998), 45--75.
\bibitem{Na2} C. Nakayama, \emph{Quasi-sections in log geometry}. Osaka Journal of Mathematics 46 (2009), 1163--1173.
\bibitem{Ni} J. Nicaise, \emph{Geometric criteria for tame ramification}. Mathematische Zeitschrift 273 (2013), 839--868.
\bibitem{Og} A. Ogus, \emph{Lectures on logarithmic algebraic geometry}. Cambridge Studies in Advanced Mathematics 178, Cambridge University Press, Cambridge, 2018.
\bibitem{Ra}
{M. Raynaud, \emph{Sp\'ecialisation du foncteur de Picard}. Publications Math\'ematiques de l'IH\'ES 38 (1970), 27--76.}

\bibitem{Sa}
{T. Saito, \emph{Vanishing cycles and geometry of curves over a discrete valuation ring}. American Journal of Mathematics 109 (1987), 1043--1085.}
\bibitem{Sa2}
{T. Saito, \emph{Log smooth extension of a family of curves and semi-stable reduction}. Journal of Algebraic Geometry 13 (2004), 287--321.}

\bibitem{Sm}
{A. Smeets, \emph{Logarithmic good reduction, monodromy and the rational volume}. Algebra \& Number Theory 11 (2017), 213--233.}

\bibitem{St} 
{J. Stix, \emph{A logarithmic view towards semistable reduction}. Journal of Algebraic Geometry 14 (2005), 119--136.}
\end{thebibliography}
\end{document}